\documentclass[11pt,a4paper,oneside,reqno]{amsart}
\usepackage{mathrsfs}
\usepackage{color}
\usepackage{amsmath,amsfonts,amssymb,amsthm,graphicx,color}
\def\({\left(}
\def\){\right)}

\newcommand{\vv}{\mathbf{v}}
\newcommand{\dd}{\dd }

\newcommand{\ee}{\mathbf{e}_3}
\newcommand{\hh}{\mathbf{H}}
\newcommand{\ff}{\mathbf{f}}
\renewcommand{\gg}{\mathbf{g}}
\newcommand{\zz}{\mathbf{z}}

\newcommand{\HH}{\mathbf{H}}
\def\l|{\left|}
\def\r|{\right|}
\newcommand{\R}{\mathbb{R}}
\renewcommand{\O}{{\Omega}}

\renewcommand{\S}{\mathbb{S}}
\renewcommand{\l}{\lambda}

\renewcommand{\o}{\omega}

\renewcommand{\a}{\alpha}
\renewcommand{\leq}{\leqslant}
\renewcommand{\geq}{\geqslant}
\newcommand{\p}{\partial}

\def\HH{\mathbf{H}}
\def\dd{\mathbf{d}}

\def\qq{\mathbf{e}}
\def\({\left (}
\def\){\right )}
\def\[{\left[}
\def\]{\right]}
\def\eps{\epsilon}

\newtheorem{theorem}{{\bf Theorem}}[section]

\newtheorem{lemma}[theorem]{{\bf Lemma}}
\newtheorem{prop}[theorem]{{\bf Proposition}}
\newtheorem{remark}[theorem]{{\bf Remark}}

\begin{document}

\title{\small Control of Harmonic Map Heat Flow with an External Field}

\date{}
\author{Yuning Liu }
\address{NYU Shanghai, 1555 Century Avenue, Shanghai 200122, China,
and NYU-ECNU Institute of Mathematical Sciences at NYU Shanghai, 3663 Zhongshan Road North, Shanghai, 200062, China}
\email{yl67@nyu.edu}

\begin{abstract}
We investigate the control problem of harmonic map heat flow  by means of an external magnetic field.
In contrast to the situation of a  parabolic  system with internal or boundary control, the   magnetic field acts as the coefficients of  lower order terms  of the equation.
We show that for initial data whose image stays in a hemisphere, with one control acting on a subset of the domain plus one that only depends on time, the state of the system can be steered to any ground state, i.e. any given unit vector, within any short time. To achieve this, in the first step  a   control is applied  to steer the solution into a small neighborhood of the peak of the hemisphere. Then under  stereographic projection, the original system is reduced  to an internal parabolic  control system with initial data sufficiently close to $0$ so  that the existing method for  local controllability can be applied. The key process  is to give an explicit solution of an underdetermined algebraic system such that the affine type control can be converted into an internal control.
\end{abstract}
\maketitle

\section{Introduction}
\setcounter{equation}{0}

We investigate the controllability of the following system with Neumann boundary condition
  \begin{equation}\label{gradientflow}
\left\{
\begin{array}{rll}
 \p_t\dd-\Delta \dd &=|\nabla \dd |^2\dd + (\mathbf{H}\cdot \dd )\mathbf{H}- (\mathbf{H}\cdot \dd )^2\dd,&~\text{in}~Q=\O\times (0,T),\\
\tfrac{\p \dd}{\p \nu}&=0,&~\text{on}~\Sigma=\p\O\times (0,T),
\end{array}
\right.
  \end{equation}
  where  $\O\subset\R^3$ is an open set with $C^2$ convex boundary $\p\O$, and $\nu$ denotes  the inner unit normal vector to   $\p\O$.
The work  is motivated by the analysis and optimal control of a simplified Ericksen--Leslie system describing the dynamics of a liquid crystal (see \cite{Lin2010}, \cite{MR3518239}  and \cite{MR3621817}) when the hydrodynamic effects are neglected. We recall that  the   mathematical description of the static configuration of liquid crystal material under a magnetic field is to consider the  Oseen--Frank model \cite{HardtKinderlehrerLin1986}. In the simplest case, the energy functional of such model has the form (see \cite{lin2007magnetic})
 \begin{equation}\label{ofma}
   \mathcal{E}(\dd )=\frac 12\int_\O\(|\nabla \dd |^2-
    (\HH\cdot \dd )^2\)dx,
 \end{equation}
 where $\dd :\O\to \S^2$ describes the local orientation  of the liquid crystal molecules, and $\mathbf{H}:\O\to \R^3$ denotes the external magnetic  field. Here we omit the diamagnetic susceptibility constant in front of the term
   $(\HH\cdot \dd )^2$. The orientation $\dd$ tends to align along the magnetic field $\HH$ for the sake of minimizing the total energy \eqref{ofma}.
By introducing a Lagrange multiplier  to penalize  the constraint $|\dd|=1$, we can derive the Euler--Lagrange equation of \eqref{ofma}
 \begin{equation}\label{stationary}
   -\Delta \dd =|\nabla \dd |^2\dd + (\mathbf{H}\cdot \dd )\mathbf{H}-(\mathbf{H}\cdot \dd )^2 \dd,
 \end{equation}
 and thus \eqref{gradientflow} is the corresponding gradient flow.

 To clarify the constraint $|\dd(x,t)|=1$ in \eqref{gradientflow}, we note   that in case the external field $\HH$ is given and is regular enough, say $C^1$ up to the boundary, then the short time  classical solution to \eqref{gradientflow} can be constructed using standard parabolic theory, provided    the  initial data $\dd(x,0):\O\mapsto \S^2$
 is  regular enough. So the scalar function
 $g(x,t)\triangleq|\dd(x,t)|^2-1$ satisfies  the   Neumann boundary condition $\tfrac{\p g}{\p \nu}=0$ and the following linear parabolic equation
 \begin{equation}\label{unique}
\p_t g-\Delta g=2(|\nabla\dd|^2-(\HH\cdot\dd)^2) g.
 \end{equation}
 So  the constraint $|\dd|^2=1$ shall be preserved as long as the classical solution exists due to the uniqueness of solution of \eqref{unique}. However, in the control problem the vector field $\HH$ is part of the unknown. In the sequel, such a constraint will be achieved through an alternative way  using the stereographic  projection.

Another feature of \eqref{gradientflow} is that it is rotationally invariant. More precisely, if $(\dd,\HH)$ satisfies \eqref{gradientflow}, so does $(\mathcal{R}\dd,\mathcal{R}\HH)$, for any orthogonal matrix $\mathcal{R}$.  This property will be used in the proof of the main theorem.

 In the last three decades there has been an enormous amount of progresses concerning  harmonic heat flow, see the comprehensive monograph \cite{LinWang2008}. A closely related work  is  due to  Chen  \cite{chen1999maximum}, who  considers the system with another form of external field compared with \eqref{gradientflow}, and discusses  the existence of a classical solution and its large time behavior  when the initial data lies in a hemisphere. However, to the best of our knowledge, there has been no result concerning the controllability of such system, neither by boundary control nor by magnetic field. On the other hand, there have been numerous advances in the controllability of nonlinear parabolic equation or system. The readers can refer to, for instance, \cite{doubova2002controllability,fernandez2006exact,MR2257228}.

The main result of this work is stated as follows.
\begin{theorem}\label{mainthm}
Let $\omega$  be an open proper subset of $\Omega$, and $\alpha\in (0,1),T>0$ be  fixed numbers.
For   any initial state  $\dd_0\in C^{2+\a}(\overline{\O},\S^2)$   satisfying $\p_\nu \dd_0=0$ on $\p\Omega$ and  \begin{equation}\label{hemi}
 \inf_{x\in\O}\dd_0(x)\cdot\qq> 0~\text{for some}~ \qq\in\S^2,
 \end{equation}
 and  any   constant state   $\mathbf{p}\in\S^2$,
 there exist $\hh_0(x,t)\in L^\infty(Q)$  and   $\mathbf{g}(t)\in L^\infty(0,T)$ such that the system \eqref{gradientflow} with initial data $\dd\mid_{t=0}=\dd_0$ and control
\begin{equation}\label{eq:1.6}
\HH(x,t)=\mathbf{g}(t) +\chi_\omega\hh_0(x,t)
\end{equation}
 satisfies $\dd\mid_{t=T}=\mathbf{p}$. In \eqref{eq:1.6} $\chi_\omega$ denotes the characteristic function of $\omega$.
 \end{theorem}
\begin{remark}
To steer the system \eqref{gradientflow} to any ground state $\mathbf{p}\in\S^2$, we first drive it to $\qq$. This can be done by   first choosing $\HH=\lambda(t)\qq$ with $\lambda$ being sufficiently large such that it forces the solution $\dd$ to stay in a small neighborhood of $\qq$ within $[0,\tfrac T8]$.   Then we construct $\HH=\chi_\omega \HH_0$  by proving a local controllability result within $[\tfrac T8,\tfrac{T}4]$ such that $\dd(\cdot,\tfrac{T}4)\equiv \qq$.   Using the rotational invariance of \eqref{gradientflow}, we can  repeat the previous process to steer the system successively to two intermediate states $\mathbf{p}_1,\mathbf{p}_2\in\S^2$ which trisect the angle between $\qq$ and $\mathbf{p}$.  Finally we drive the state from $\mathbf{p}_2$ to  $\mathbf{p}$.
\end{remark}

The rest of the paper is  organized as follows. In Section 2, we recast \eqref{gradientflow} into a semi-linear parabolic system with internal control.  In Section 3, we prove the existence and uniqueness of the global in time classical solution to \eqref{gradientflow} with a special choice of the magnetic field, i.e.  $\HH=\lambda\qq$. This result is based on a  Bernstein type estimate and the novelty is that the Lipschitz norm of the solution is independent of the size of $\lambda(t)$. Based on the results in these two sections, we  give the proof of Theorem \ref{mainthm}    in the last section.

Regarding notation, we shall use bold letters to  denote vectors or matrices, and use  the non-bold  letters with indices to denote their components. For instance, $\dd=(d_1,d_2,d_3)=(d_i)_{1\leq i\leq 3}$. We shall adopt  the convention in differential geometry that the partial derivatives $\p_{x_i}$  of various tensors are abbreviated by adding $_{,i}$  to the corresponding components: $\p_{x_i}d_j=d_{j,i}$. Moreover, repeated indices will be summed. The standard basis vectors in $\mathbb{R}^3$ are denoted by $\mathbf{e}_i$ with $1\leq i\leq 3$. We shall use $\mathbf{a}\cdot\mathbf{b}=a_ib_i$  for the inner product   and colon for the contraction of two matrices $\mathbf{A}:\mathbf{B}=A_{ij}B_{ij}$.

 \section{Reduction  to parabolic  system with internal control}
\setcounter{equation}{0}
In this section we shall use the stereographic projection to remove the constraint $|\dd |=1$ in \eqref{gradientflow} and  reduce it    to a  parabolic   system with internal control whose support  lies in an open  subset $\o\subsetneqq\O$.
The  stereographic projection   $\mathbf{\Psi}=(\Psi_1,\Psi_2,\Psi_3):\R^2\to \S^2\backslash\{-\ee\}$  is defined via
\begin{equation}\label{stereograph}
 \mathbf{\Psi}(v_1,v_2)\triangleq\(\frac{2v_1}{1+v_1^2+v_2^2},\frac{2v_2}{1+v_1^2+v_2^2}
  ,\frac{1-v_1^2-v_2^2}{1+v_1^2+v_2^2}\).
\end{equation}

\begin{prop}\label{equiv3}
Let $\dd$ be a  classical solution to  \eqref{gradientflow} satisfying
\begin{equation}
\inf_{\O\times (0,T)}|\dd(x,t)+\ee |>0,
\end{equation}
then $\vv=\mathbf{\Psi}^{-1}(\dd)$ is a classical solution to the following equation
  \begin{equation}\label{hfstereo1}
  \left\{
\begin{array}{rll}
\p_t \mathbf{v}&=\Delta \mathbf{v} -2\nabla \mathbf{v}\cdot\nabla \log h +\frac{2|\nabla \mathbf{v}|^2}{h} \mathbf{v}
 +\frac{h^2}{4}(\mathbf{H}\cdot \dd )H_i \nabla_\vv \Psi_i(\vv),&~\text{in}~Q,\\
\tfrac{\p \vv}{\p\nu}&=0,&~\text{on}~\Sigma,
\end{array}
\right.
\end{equation}
where $h=1+|\vv|^2$.
 Conversely, if $\vv$ is a strong solution to \eqref{hfstereo1}, then $\dd=\mathbf{\Psi}(\vv)$ is a strong solution to \eqref{gradientflow}.
\end{prop}
\begin{proof}  It follows from \eqref{stereograph} and $\dd=\mathbf{\Psi}(\vv)$ that
  $\nabla\dd=\(\frac{\p d_i}{\p x_k}\)_{1\leq i,k\leq 3}$ and $\p_t\dd$ can be computed by
\begin{equation}\label{eq:2.0}
\quad\frac{\p d_i}{\p x_k}=\sum_{j=1}^2\frac{\p \Psi_i}{\p v_j}\frac{\p v_j}{\p x_k},\quad
    \frac{\p d_i}{\p t}=\sum_{j=1}^2\frac{\p \Psi_i}{\p v_j}\frac{\p v_j}{\p t}.
\end{equation}
As a result,
\begin{equation}\label{eq:2.1}
  \begin{split}
    |\nabla\dd |^2&=\operatorname{tr}\(\nabla\dd \(\nabla\dd \)^T\)=\sum_{\ell,k=1}^3\sum_{j,s=1}
    ^2\frac{\p v_j}{\p x_k}\(\frac{\p \Psi_\ell}{\p v_j}\frac{\p \Psi_\ell}{\p v_s}\)\frac{\p v_s}{\p x_k},\\
  \Delta\dd &=\frac{\p}{\p x_k}\frac{\p d_i}{\p x_k}=\frac{\p}{\p x_k}\(\sum_{j=1}^2\frac{\p \Psi_i}{\p v_j}\frac{\p v_j}{\p x_k} \).
  \end{split}
\end{equation}
Denote
\begin{equation}\label{aijdeter}
 A_{jk}(\vv)\triangleq \frac{\p \Psi_i}{\p v_j} \frac{\p \Psi_i}{\p v_k},~\text{for}~ 1\leq j,k\leq 2,
\end{equation}

\begin{equation}\label{largeJ}
 \mathbf{J}\triangleq-\p_t \dd+\Delta \dd +|\nabla \dd |^2\dd +(\mathbf{H}\cdot \dd )\mathbf{H}-(\mathbf{H}\cdot \dd )^2 \dd,
\end{equation}
and
\begin{equation} \label{largeM}
 \mathbf{M}\triangleq-\p_t \vv+\Delta \mathbf{v} -2\nabla \mathbf{v}\cdot\nabla \log h +\frac{2|\nabla \mathbf{v}|^2}{h} \mathbf{v}
 +\frac{h^2}{4}(\mathbf{H}\cdot \dd )H_i\cdot\nabla_\vv \Psi_i(\vv),
\end{equation}
with $h=1+|\mathbf{v}|^2$, $\mathbf{J}=\{J_i\}_{1\leq i\leq 3}$ and $\mathbf{M}=\{M_i\}_{1\leq i\leq 2}$. Then we need to show the     following equivalence:
\begin{equation}\label{equiv4}
  \mathbf{M}=0\Longleftrightarrow \mathbf{J}=0.
\end{equation}
To do this, we first use \eqref{eq:2.0} and \eqref{eq:2.1} to write $\mathbf{J}$ component-wise
\begin{equation*}
  J_i=-\frac{\p \Psi_i}{\p v_j}\frac{\p v_j}{\p t}+\frac{\p}{\p x_k}\(\frac{\p \Psi_i}{\p v_j}\frac{\p v_j}{\p x_k} \)+d_i\frac{\p v_j}{\p x_k}\(\frac{\p \Psi_\ell}{\p v_j}\frac{\p \Psi_\ell}{\p v_s}\)\frac{\p v_s}{\p x_k}+(\mathbf{H}\cdot \dd )H_i-(\mathbf{H}\cdot \dd )^2 d_i.
\end{equation*}
Multiplying the above equality by $\frac{\p \Psi_i}{\p v_{\ell}}$, summing over $i$ and using $|\dd|=1$, we obtain
\begin{equation}\label{leftmul}
\begin{split}
   J_i\frac{\p \Psi_i}{\p v_{\ell}} =&-\frac{\p \Psi_i}{\p v_{\ell}} \frac{\p \Psi_i}{\p v_j}\frac{\p v_j}{\p t}+\frac{\p}{\p x_k}\(\frac{\p \Psi_i}{\p v_{\ell}}\frac{\p \Psi_i}{\p v_j}\frac{\p v_j}{\p x_k} \)\\
   &-\frac{\p}{\p x_k}\(\frac{\p \Psi_i}{\p v_{\ell}}\)\(\frac{\p \Psi_i}{\p v_j}\frac{\p v_j}{\p x_k} \)+(\mathbf{H}\cdot \dd )H_i\frac{\p \Psi_i}{\p v_{\ell}}\\
   =&-A_{{\ell}j}(v)\frac{\p v_j}{\p t}+\frac{\p A_{{\ell}j}(v)}{\p x_k}\frac{\p v_j}{\p x_k} +A_{{\ell}j}(v)\Delta v_j\\
   & -\frac{\p}{\p x_k}\(\frac{\p \Psi_i}{\p v_{\ell}}\)\(\frac{\p \Psi_i}{\p v_j}\frac{\p v_j}{\p x_k} \)+(\mathbf{H}\cdot \dd )H_i\frac{\p \Psi_i}{\p v_{\ell}}.
\end{split}
\end{equation}
In the second equality above we employed  \eqref{aijdeter}. On the other hand, it follows from \eqref{stereograph}
 that
\begin{equation}\label{jacobian}
  \frac{\p \Psi_i}{\p v_j}=
  \begin{pmatrix}
  \frac{2}{1+v_1^2+v_2^2}-\frac{4v_1^2}{(1+v_1^2+v_2^2)^2}&-\frac{4v_1v_2}{(1+v_1^2+v_2^2)^2}\\
  -\frac{4v_1v_2}{(1+v_1^2+v_2^2)^2}&\frac{2}{1+v_1^2+v_2^2}-\frac{4v_2^2}{(1+v_1^2+v_2^2)^2}\\
  -\frac{4v_1}{(1+v_1^2+v_2^2)^2}&-\frac{4v_2}{(1+v_1^2+v_2^2)^2}
  \end{pmatrix}.
\end{equation}
Recalling that $h=1+v_1^2+v_2^2$, we have a precise formula of \eqref{aijdeter},
\begin{equation}\label{aijformu}
    A_{{\ell}j}(\vv)=\frac{4}{h^2}\delta_{{\ell}j}.
\end{equation}
This  simplifies  \eqref{leftmul}  into
\begin{equation}\label{leftmu2}
 \begin{split}
    J_i\frac{\p \Psi_i}{\p v_{\ell}}  =&-\frac{4}{h^2}\delta_{{\ell}j}\frac{\p v_j}{\p t}+\frac{\p }{\p x_k}\(\frac{4}{h^2}\)\delta_{{\ell}j}\frac{\p v_j}{\p x_k} +\frac{4}{h^2}\delta_{{\ell}j}\Delta v_j\\
   &-\frac{\p}{\p x_k}\(\frac{\p \Psi_i}{\p v_{\ell}}\)\(\frac{\p \Psi_i}{\p v_j}\frac{\p v_j}{\p x_k} \)+(\mathbf{H}\cdot \dd )H_i\frac{\p \Psi_i}{\p v_{\ell}}.
 \end{split}
\end{equation}
To proceed, we denote
\begin{equation*}
B_{j\ell s}\triangleq\frac{\p^2 \Psi_i}{\p v_{\ell}\p v_s}\frac{\p \Psi_i}{\p v_j}.
\end{equation*}
Notice that
\begin{equation*}
  B_{j\ell s}+B_{s\ell j}=\frac{\p \Psi_i}{\p v_j}\frac{\p^2 \Psi_i}{\p v_{\ell}\p v_s}+\frac{\p \Psi_i}{\p v_s}\frac{\p^2 \Psi_i}{\p v_{\ell}\p v_j}=\frac{\p A_{sj}(\vv)}{\p v_{\ell}}.
\end{equation*}
By a permutation,
\begin{equation*}
 \begin{split}
    B_{j\ell s}=&\frac 12\(\frac{\p A_{sj}(\vv)}{\p v_\ell}+\frac{\p A_{j\ell}(\vv)}{\p v_s}-\frac{\p A_{\ell s}(\vv)}{\p v_j}\)
    =-\frac{4}{h^3}\(h_{,\ell}\delta_{sj}+h_{,s}\delta_{j\ell}-h_{,j}
    \delta_{\ell s}\),
 \end{split}
\end{equation*}
where   $h_{,\ell}$ is the abbreviation for $\frac{\p h}{\p v_{\ell}}=2v_\ell$.
Applying this formula to the fourth component of the right hand side in \eqref{leftmu2} gives
\begin{equation}\label{trinonline}
  \begin{split}
  -\frac{\p}{\p x_k}\(\frac{\p \Psi_i}{\p v_{\ell}}\)\(\frac{\p \Psi_i}{\p v_j}\frac{\p v_j}{\p x_k} \)&=-B_{j\ell s}\frac{\p v_j}{\p x_k}\frac{\p v_s}{\p x_k}\\
  &=\frac{4}{h^3}\(h_{,\ell}v_{s,k}v_{s,k}+h_{,s}v_{\ell,k}
  v_{s,k}-h_{,j}v_{j,k}v_{\ell, k}\)\\
  &=\frac{4}{h^3}h_{,\ell}|\nabla\vv|^2,  \end{split}
\end{equation}
where $v_{i,j}$ is the abbreviation of $\frac{ \p v_i}{\p x_j}$.
Plug \eqref{trinonline} into \eqref{leftmu2} to get
\begin{equation*}
 \begin{split}
   J_i\frac{\p \Psi_i}{\p v_{\ell}} =&-\frac{4}{h^2}\p_t v_{\ell}-\frac{8}{h^3}\nabla h\cdot\nabla v_{\ell} +\frac{4}{h^2}\Delta v_{\ell}+\frac{8}{h^3}v_{\ell}|\nabla\vv|^2 +(\mathbf{H}\cdot \dd )H_i\frac{\p \Psi_i}{\p v_{\ell}}.
 \end{split}
\end{equation*}
In  virtue of  \eqref{largeM}, this is equivalent to
  \begin{equation*}
    \begin{pmatrix}
      \frac{\p \Psi_1}{\p v_1}& \frac{\p \Psi_2}{\p v_1} & \frac{\p \Psi_3}{\p v_1}\\
      \\
      \frac{\p \Psi_1}{\p v_2}& \frac{\p \Psi_2}{\p v_2} & \frac{\p \Psi_3}{\p v_2}
    \end{pmatrix}\begin{pmatrix}
      J_1\\
      J_2\\
      J_3
    \end{pmatrix}=\frac 4{h^2}\begin{pmatrix}
      M_1\\
      M_2
    \end{pmatrix}.
  \end{equation*}
Note that $\dd\cdot\mathbf{J}\equiv0$, due to $|\dd|\equiv1$,
the above formula is equivalent to
\begin{equation}\label{equiv5}
  \mathbf{E}\mathbf{J}=\frac 4{h^2}(0,M_1,
  M_2)^T,
\end{equation}
 where $\mathbf{E}$ is  the $3\times 3$ matrix
    \begin{equation*}
    \mathbf{E}\triangleq\begin{pmatrix}
    d_1&d_2&d_3\\
    \\
      \frac{\p \Psi_1}{\p v_1}& \frac{\p \Psi_2}{\p v_1} & \frac{\p \Psi_3}{\p v_1}\\
      \\
      \frac{\p \Psi_1}{\p v_2}& \frac{\p \Psi_2}{\p v_2} & \frac{\p \Psi_3}{\p v_2}
    \end{pmatrix}.
  \end{equation*}
 As a result \eqref{equiv4} is a consequence of  $\det \mathbf{E}\neq 0$. Actually, using \eqref{aijformu},
  \begin{equation*}
    (\det \mathbf{E})^2=\det (\mathbf{E}\mathbf{E}^T)=\det\begin{pmatrix}
      1&0&0\\
      0&A_{11}&A_{12}\\
      0&A_{21}&A_{22}
    \end{pmatrix}^2=(2/h)^8.  \end{equation*}
Concerning the boundary condition, we have
   \begin{equation*}
    \begin{pmatrix}
    \p_\nu d_1\\
    \p_\nu d_2\\
    \p_\nu d_3
    \end{pmatrix}
=\begin{pmatrix}
    \frac{\p \Psi_1}{\p v_1} &\frac{\p \Psi_1}{\p v_2}   \\
    \\
      \frac{\p \Psi_2}{\p v_1} &\frac{\p \Psi_2}{\p v_2} \\
      \\
        \frac{\p \Psi_3}{\p v_1} &\frac{\p \Psi_3}{\p v_2}
     \end{pmatrix} \begin{pmatrix}
    \p_\nu v_1\\
    \\
    \p_\nu v_2\\
    \end{pmatrix}=\begin{pmatrix}
    \frac{\p \Psi_1}{\p v_1} &\frac{\p \Psi_1}{\p v_2}  & d_1 \\
    \\
      \frac{\p \Psi_2}{\p v_1} &\frac{\p \Psi_2}{\p v_2}& d_2 \\
      \\
        \frac{\p \Psi_3}{\p v_1} &\frac{\p \Psi_3}{\p v_2} & d_3
     \end{pmatrix} \begin{pmatrix}
    \p_\nu v_1\\
    \\
    \p_\nu v_2\\
    \\
    0
    \end{pmatrix}.
  \end{equation*}
  This together with $\det \mathbf{E}\neq 0$ implies the equivalence between boundary conditions $\p_\nu \dd=0$ and $\p_\nu \vv=0$.
  So we complete the proof.
\end{proof}


In order to reduce \eqref{gradientflow} to an internal  control system, we  write the last component of \eqref{hfstereo1} as $\chi_\omega \ff$, where $\chi_\o$ is the  characteristic function of an open subset $\o\subsetneqq\O$:
\begin{equation}
\frac{h^2}{4}(\mathbf{H}\cdot \dd )H_i \nabla_\vv \Psi_i(\vv)=\chi_\omega \ff.
\end{equation}
 In view of \eqref{jacobian},   this amounts to  solving the following algebraic equations of $\HH$ for given  $\vv=(v_1,v_2)$ and $\ff=(f_1,f_2)$:
 \begin{equation}\label{algebraic}
   \frac{2v_1H_1+2v_2 H_2+(1-v_1^2-v_2^2)H_3}{1+v_1^2+v_2^2}\begin{pmatrix}
    \frac12 (1+v_2^2-v_1^2)H_1-v_1v_2 H_2- v_1H_3\\
    \\
    -v_1v_2 H_1+\frac 12(1+v_1^2-v_2^2)H_2-v_2H_3
  \end{pmatrix}=\chi_\o\begin{pmatrix}
    f_1\\
    f_2
  \end{pmatrix}.
 \end{equation}

 \begin{lemma}\label{implicitfunc}
  For every $(\vv,\ff)\in C(\O;\mathbb{R}^4)$, equation \eqref{algebraic}  has a   solution $\HH=\HH(\ff,\vv)$   which depends analytically on $\vv$ and $\ff$ such that $\operatorname{supp}(\HH)\subset \o$.
 \end{lemma}
 \begin{proof}
  The equation \eqref{algebraic} is underdetermined and might have multiple solutions. We look for a special solution  by setting
 $$
2v_1H_1+2v_2H_2+(1-v_1^2-v_2^2)H_3=(1+v_1^2+v_2^2)\chi_\o.
 $$
 Then \eqref{algebraic} can be reduced to the following linear equation about $H_i$
 \begin{equation}\label{algebraic-1}
   \begin{pmatrix}
    \frac12 (1+v_2^2-v_1^2)H_1-v_1v_2 H_2- v_1H_3\\
    \\
    -v_1v_2 H_1+\frac 12(1+v_1^2-v_2^2)H_2-v_2H_3\\ \\
    -v_1H_1-v_2H_2+\frac 12(-1+v_1^2+v_2^2)H_3
  \end{pmatrix}=\chi_\o\begin{pmatrix}
    f_1\\ \\
    f_2\\ \\
    -\frac 12(1+v_1^2+v_2^2)
  \end{pmatrix}.
 \end{equation}
Denote
 $$
 \begin{array}{ll}
 \displaystyle \mathbf{A}=\left(\begin{array}{cccc}
   & \frac12 (1+v_2^2-v_1^2) &-v_1v_2 &-v_1\\
    &-v_1v_2 & \frac12(1+v_1^2-v_2^2)&-v_2\\
    &-v_1&-v_2&\frac 12(-1+v_1^2+v_2^2)
  \end{array}\right).
  \end{array}
  $$
 Its eigenvalues   are
  \begin{equation*}
    \begin{split}
      \lambda_1&=-\tfrac 12(1+v_1^2+v_2^2),\\
      \lambda_2&=\lambda_3=\tfrac 12(1+v_1^2+v_2^2),
    \end{split}
  \end{equation*}
  and its cofactor matrix consists of entries that are polynomials of $v_1$ and $v_2$. Thus $\mathbf{A}$ is invertible, and $\mathbf{A}^{-1}$ depends analytically on $(v_1,v_2)$.
  This shows that \eqref{algebraic-1}  has a unique analytic  solution and the lemma is proved.
 \end{proof}

Thanks to Proposition \ref{equiv3} and Lemma \ref{implicitfunc}, the controllability of \eqref{gradientflow} is reduced to the following system with internal control:
   \begin{equation}\label{noncontrolstere}
        \left\{
        \begin{array}{rll}
          \p_t \vv-\Delta \vv&=-2\nabla \mathbf{v}\cdot\nabla \log (1+|\vv|^2) +\frac{2|\nabla \mathbf{v}|^2}{1+|\vv|^2} \mathbf{v}+ \chi_\o  \ff, &\text{in}~Q,\\
          \vv|_{t=0} &=\vv_0,&\text{in}~ \O,\\
        \p_\nu \vv&=0,&\text{on}~\Sigma.
        \end{array}
        \right.
     \end{equation}

The local controllability of \eqref{noncontrolstere} is actually a consequence of  \cite{fernandez2006exact}. For the convenience of the readers, we give the proof based on the  following  result:
 \begin{lemma}\label{linearcontrol}
Let $\mathbf{a}(x,t)\in L^\infty(Q;\mathbb{R}^{2\times 2})$ and  $\mathbf{y}_0\in L^2(\Omega;\mathbb{R}^2)$.
For every $T>0$,  the system
\begin{equation}\label{linearization1}
        \left\{
        \begin{array}{rll}
          \p_t \mathbf{y}-\Delta \mathbf{y}&=\mathbf{a}(x,t)\mathbf{y} + \chi_\o  \mathbf{u}, &\text{in}~Q,\\
          \mathbf{y}|_{t=0} &=\mathbf{y}_0,&\text{in}~ \O,\\
        \p_\nu \mathbf{y}&=0,&\text{on}~\Sigma,
        \end{array}
        \right.
     \end{equation}
 is null-controllable at $t=T$. Moreover,  the control $\mathbf{u}\in L^\infty(Q;\mathbb{R}^2)$ satisfies
\begin{equation}
\| \mathbf{u}\|_{L^\infty(Q)}\leq e^{c_0(\O,\omega)K\(T,\|\mathbf{a}\|_{L^\infty(Q)}\)}\| \mathbf{y}_0\|_{L^2(\O)},
\end{equation}
where $c_0(\O,\omega)>0$ is a generic constant and $K\(T,\|\mathbf{a}\|_{L^\infty(Q)}\)>0$ is non-decreasing on its second argument.

\end{lemma}
The scalar version of the above result is proved in \cite{fernandez2006exact} using Carleman estimate. The proof can be directly adapted to the above vectorial case because the control is applied to every component of the system in \eqref{linearization1}.
To proceed, we need the   Kakutani's fixed point theorem. See for instance \cite[Chapter 2, Section 5.8]{MR1987179} or \cite[pp. 126]{MR1157815} for the proof.
\begin{prop}[Kakutani's fixed point theorem]\label{kaku prop}
Let $\mathcal{Z}$ be a non-empty, compact and convex subset of a Hausdorff locally convex topological vector space $\mathcal{Y}$, and  $2^{\mathcal{Z}}$ be its power set, i.e. the set of all subsets of $\mathcal{Z}$. Let $\Phi: \mathcal{Z}\to 2^\mathcal{Z}$ be  upper semi-continuous and $\Phi(x)$ is non-empty, compact, and convex for all $x\in \mathcal{Z}$. Then $\Phi$ has a fixed point in the sense that there exists $x\in \mathcal{Z}$ such that $x\in \Phi(x)$.
\end{prop}
We also need a few results concerning parabolic regularity:
\begin{prop}\label{lemma semigroup}
Let $N$ be the dimension of $\Omega$ and $u$ be a solution of
 \begin{equation}
\left\{
\begin{array}{rll}
\p_t u-\Delta u&=g,&~\text{in}~\O\times (0,T),\\
\p_\nu u&=0,&~\text{on}~\p\O\times (0,T).
\end{array}
\right.
  \end{equation}
 Then for any  $p\in (1,\infty)$, we have
 \begin{equation}\label{heat est2}
   \|u\|_{L^p(0,T;W^{2,p}(\O))}+\|\p_t u\|_{L^p(0,T;L^p(\O))}
  \leq   C(\O,p)\(\|u|_{t=0}\|_{W^{2,p}(\O)}+\|g\|_{L^p(\O\times (0,T))}\),
 \end{equation}
and
\begin{equation}\label{heat est1}
  \|u\|_{C([0,T];W^{1,\infty}(\O))}\leq C(\Omega,p)\(\|u|_{t=0}\|_{W^{1,\infty}(\O)}+\|g\|_{L^p(\O\times (0,T))}\),~\forall p\in (N+2,\infty).
\end{equation}

\end{prop}
\begin{proof}
Estimate \eqref{heat est2} follows from  parabolic theory, see for instance \cite{MR2318575}. Note that the regularity needed for the initial data is far from being optimal but is sufficient for later use.  Regarding  estimate \eqref{heat est1}, the case when $u|_{t=0}=0$  is a consequence of \eqref{heat est2} and Sobolev embedding. The case when $g=0$ follows from
the following estimate of  the analytic semigroup $e^{t\Delta}$ generated by the Neumann-Laplacian
\begin{equation}
\|e^{t\Delta} f\|_{L^{q}(\O)}\leq C(p,q)t^{-\frac N2\(\frac 1{p}-\frac 1{q}\)}\|f\|_{L^{p}(\O)},  \quad  \forall 1\leq p\leq q\leq \infty;~t\in (0,1).\label{semi-group}
\end{equation}
For the proof of  \eqref{semi-group},  we first recall from \cite[page 90] {MR1103113} or \cite{MR834612} that  the  heat kernel $K(t,x,y)$ of the Neumann-Laplacian   satisfies
\begin{equation}\label{kernel est}
C_1 t^{-\frac N 2}e^{-\frac{(x-y)^2}{4t}}\leq K(t,x,y)\leq C_2  t^{-\frac N 2}e^{-\frac{(x-y)^2}{8t}},\qquad \forall x,y\in\O~\text{and}~t\in (0,1).
\end{equation}
Since $e^{t\Delta}f(x)=\int_{\Omega} K(t,x,y) f(y)\,dy$, the  inequality \eqref{kernel est} together with  the Young's inequality of convolution leads to \eqref{semi-group}.
 \end{proof}

The combination of the previous  results leads to  the local controllability of \eqref{noncontrolstere}.
\begin{prop}\label{conlinear}
For every $T>0$, there exists a   constant $\epsilon_2=\epsilon_2(T)>0$ such that if
\begin{equation}\label{initialdata}
\|\vv_0\|_{W^{1,\infty}(\O)}\leq \min\{1,\epsilon_2\}~\text{and}~   \vv_0\in W^{2,p}(\O),~\text{for some fixed}~p>5,
\end{equation}
then the system \eqref{noncontrolstere} is null-controllable at time $T$.
\end{prop}

\begin{proof}
Since we are only concerned with short time controllability, without loss of generality, we  assume  $T\in (0,1)$.  We shall employ the Kakutani's fixed point theorem (in Proposition \ref{kaku prop}) to show the null-controllability.  To proceed, we choose $R>1$ and  introduce
\begin{equation}\label{trajectory}
       \mathcal{Z}\triangleq\left\{\zz\in C([0,T];W^{1,\infty}(\O))\mid \|\zz\|_{C([0,T];W^{1,\infty}(\O))}\leq R,~\zz(x,0)=\vv_0(x)\right\}.
     \end{equation}
 Then it follows from \eqref{initialdata} that $\mathcal{Z}$ is a nonempty convex and compact subset of some negative Sobolev space, say $H^{-1}(Q)$. Given $\zz\in \mathcal{Z}$, consider the linear  control system
\begin{equation}\label{linearcon}
  \left\{
  \begin{array}{rll}
    \p_t \vv-\Delta \vv&=\gg(\zz,\nabla \zz)\vv+  \chi_\o \ff,\quad &\text{in}~Q,\\
    \vv|_{t=0}&=\vv_0,\quad &\text{on}~\O,\\
        \p_\nu \vv&=0,&\text{on}~\Sigma,
     \end{array}
  \right.
\end{equation}
where $\gg$ is a  $2\times 2$ matrix-valued function    $$\gg(\vv,\nabla \vv)\triangleq\left\{\frac{-4\nabla v_i\cdot\nabla v_j+2|\nabla \vv|^2\delta_{ij}}{1+|\vv|^2}\right\}_{1\leq i,j\leq 2}.$$
In view of \eqref{trajectory}, we have  $\|\gg(\zz,\nabla\zz)\|_{L^\infty(Q)}\leq 8 R^2$. (Here the precise bound is not important.)
So we can   solve the linear system  \eqref{linearcon} and obtain a  control $\ff$ in the class
 \begin{equation}\label{eq:1.1}
   \mathcal{F}\triangleq\left\{\ff\in L^\infty(Q)\mid\|\ff\|_{L^\infty(Q)}\leq e^{c_0(\O,\omega)K(T,8R^2)}
   \|\vv_0\|_{L^2(\O)}\right\},
 \end{equation}
 where the constants $c_0$  and $K$ above were first introduced in Lemma \ref{linearcontrol}.
 It follows from Lemma \ref{linearcontrol} that, for every $\zz\in \mathcal{Z}$, there exists $\ff\in\mathcal{F}$ such that the system \eqref{linearcon} satisfies $\vv(\cdot,T)=0$. In other words, for every $\zz\in \mathcal{Z}$, the following set is not empty:
 \begin{equation}\label{control}
   \mathcal{C}(\zz)\triangleq \left\{\ff\quad\bigg|\quad\begin{split}
   & \|\ff\|_{L^\infty(Q)}\leq  e^{c_0(\O,\omega)K(T,8R^2)}
   \|\vv_0\|_{L^2(\O)},\\
   &\text{such that the solution to \eqref{linearcon} satisfies}~\vv(\cdot,T)=0
   \end{split}
\right\}.
 \end{equation}
 Moreover, combining  \eqref{eq:1.1}, $\|\zz\|_{L^\infty(0,T;W^{1,\infty}(\Omega))}\leq R$  and Proposition \ref{lemma semigroup}, we infer that there are two positive constants $C_1,C_2$ depending on $\Omega$   such that
\begin{equation}\label{parabolic6}
  \begin{split}
   \|\vv\|_{C([0,T];W^{1,\infty}(\O))}&\leq C_1\(\|\vv_0\|_{W^{1,\infty}(\O)}+\|\ff\|_{L^p(Q)}\)\\
   &\leq C_1\(\|\vv_0\|_{W^{1,\infty}(\O)}+e^{c_0(\O,\omega)K(T,8R^2)}
   \|\vv_0\|_{L^2(\O)}\),
  \end{split}
 \end{equation} with $p>5$, and
 \begin{equation}\label{parabolic2}
\begin{split}
  \|\vv\|_{L^p(0,T;W^{2,p}(\O))}+\|\p_t \vv\|_{L^p(0,T;L^p(\O))}&\leq C_2\(\|\vv_0\|_{W^{2,p}(\O)}+\|\ff\|_{L^p(Q)}\)\\
  &\leq C(\omega,\O,R,\vv_0).
\end{split}
 \end{equation}
  So for fixed $R>1$, by choosing a sufficiently small $\epsilon_2>0$ in \eqref{initialdata}, we infer from \eqref{parabolic6} that
  \begin{equation}\label{parabolic3}
  \begin{split}
   &\|\vv\|_{C([0,T];W^{1,\infty}(\O))} < R.
  \end{split}
 \end{equation}
Thanks to \eqref{initialdata} and \eqref{parabolic3}, we can define  a multi-valued map $\Phi:\mathcal{Z}\to 2^\mathcal{Z}$ by
  \begin{equation*}
    \Phi(\zz)\triangleq\left\{\vv\mid\vv~\text{is a solution of \eqref{linearcon} for some }~\ff\in\mathcal{C}(\zz)\right\}.
  \end{equation*}
It remains to  verify the hypothesis of Kakutani's fixed point theorem for  $\Phi$.  It is clear that $\mathcal{Z}$ is a closed, compact, convex subset of a negative Sobolev space. With $p>5$, the compactness of $\Phi(\zz)$ follows from \eqref{parabolic2} and the  compact embedding
  \begin{equation*}
  L^p(0,T;W^{2,p}(\O))\cap W^{1,p}(0,T;L^p(\O))\hookrightarrow C([0,T];W^{1,\infty}(\O)),
  \end{equation*}
  see e.g. \cite{simoncompact}. The continuity of $\Phi(\zz)$ follows from the linearity of \eqref{linearcon} and local continuity of   operator $\gg(\zz,\nabla\zz):W^{1,\infty}(\O)\mapsto L^\infty(\O)$. This completes the proof of the result.
 \end{proof}


\section{Classical Solution to  harmonic map heat flow}\label{regularity}
\setcounter{equation}{0}

In this section we consider  \eqref{gradientflow} with  $\hh=\l(t)\qq$ for some  $\qq\in\S^2$ and $\lambda(t)\in C^1([0,T])$, i.e.
  \begin{equation}\label{gradientflow1}
\left\{
\begin{array}{rll}
 \p_t\dd-\Delta \dd &=|\nabla \dd |^2\dd +  \p V(\dd),&~\text{in}~Q,\\
\p_\nu \dd&=0,&~\text{on}~\Sigma,
\end{array}
\right.
  \end{equation}
where  we denote
\begin{equation}\label{variationpotential}
  \p V(\dd)\triangleq (\mathbf{H}\cdot \dd )\mathbf{H}-(\mathbf{H}\cdot \dd )^2 \dd= \lambda^2(t) (\qq\cdot \dd) \qq-\lambda^2(t) (\qq\cdot\dd)^2 \dd.
 \end{equation}
  Note that \eqref{variationpotential} is the variation of $V(\dd)\triangleq-
    (\HH\cdot \dd )^2/2$ under the constraint $|\dd|=1$. With the notation $\mu(x,t)\triangleq \dd(x,t)\cdot \qq$, we have $V(\dd)=-\lambda^2(t)\mu^2/2$.

The main result of this section is given below, which is  essentially due to \cite{jager1979uniqueness} and \cite{chen1999maximum}. The  small but crucial  novelty we made
  is to show that  the gradient estimate is independent of $\lambda(t)$. Then choosing $\lambda(t)$ sufficiently large will force the solution $\dd$ to approach $\qq$ within any short time. Recall that we   assume $\p\O$ to be convex. This will be used   to handle the boundary conditions when applying the  maximum principle.
\begin{prop}\label{mainwellpose}
For arbitrary $T>0$ and $\alpha\in (0,1)$, assume $\dd_0\in C^{2+\a}(\overline{\O},\S^2)$ fulfills $\p_\nu\dd_0=0$ on the boundary $\p\O$ and
  \begin{equation}\label{bound}
 \eps_0\triangleq  \inf_{x\in\O}\dd_0(x)\cdot {\qq}>0.
  \end{equation}
  Then \eqref{gradientflow1}
   has a unique solution $\dd(x,t)\in C^{2+\a,1+\frac{\a}{2}}(\overline{\O}\times[0,T],\S^2)$ with initial data $\dd_0$. Moreover, $\dd$ satisfies
   \begin{equation}\label{eq:2.4}
        \sup_{\O\times [0,T]}|\nabla\dd(x,t)|\leq \frac 2{\eps_0}\sup_{x\in\O}|\nabla\dd_0|,
      \end{equation}
      and
\begin{equation}\label{maximum}
      \dd(x,t)\cdot \qq\geq  \eps_0,\quad \forall(x,t)\in\O\times [0,T].
    \end{equation}
\end{prop}
We start with a lemma saying that the projection of equation \eqref{gradientflow} to direction $\qq\in \S^2$ satisfies a  parabolic equation where the maximum principle applies:
\begin{lemma}\label{notflee}
   Under the assumptions of Proposition \ref{mainwellpose}, if $\dd\in C^{2+\a,1+\frac{\a}{2}}(\overline{Q},\S^2)$ is a solution to \eqref{gradientflow1}  with initial data $\dd_0$, then   \eqref{maximum} holds.
\end{lemma}
\begin{proof}

By the assumption  \eqref{bound} and  the continuity of the solution $\dd$, we know that
\begin{equation*}
\mu(x,t)=\dd(x,t)\cdot \qq> 0,\quad \forall(x,t)\in\O\times [0,\delta]
\end{equation*}
for some $\delta>0$. Since $\mu\in (0,1]$, it follows from \eqref{variationpotential} that
  \begin{equation*}
    \p_t \mu-\Delta \mu=|\nabla \dd|^2 \mu+ \l^2(t)(\mu-\mu^3)\geq 0.
  \end{equation*}
  This together with the maximum principle leads to the lower bound \eqref{maximum} for $t\in [0,\delta]$. Now using $\mu(\cdot,\delta)$ as initial data, and by the same argument,  we deduce that  \eqref{maximum} holds for $t\in [0,2\delta]$. This process can be carried out as long as the solution exists.  Let $m$ be the first integer such that $m \delta\geq T$.  After repeating this argument $m$ times, we  show that \eqref{maximum} holds for $t\in [0,T]$.
  \end{proof}
 The next result is concerned with the gradient estimate of  the solution to \eqref{gradientflow1}, which follows from a  Bernstein type estimate.
 \begin{lemma}\label{gradientest}
  Under the assumptions of Proposition \ref{mainwellpose}, if  $\dd\in C^{2+\a,1+\frac{\a}{2}}(\overline{Q},\S^2)$ is a solution to \eqref{gradientflow1}  with initial data $\dd_0$, then
\eqref{eq:2.4} holds.
\end{lemma}
\begin{proof}
We choose an arbitrary number  $\delta_0\in (0,\eps_0)$ and denote
   \begin{equation*}
 A(x,t)\triangleq\frac{e(\dd)}{f^2(\dd)},~\text{where}\qquad  e(\dd)\triangleq\frac 12|\nabla\dd(x,t)|^2,\qquad  f(\dd)\triangleq\dd(x,t)\cdot \qq-\delta_0.
   \end{equation*}
   We shall show that $A$ satisfies a parabolic inequality to which the maximum principle applies and yields   a bound on the gradient of $\dd$.
  We first deduce  from Lemma \ref{notflee}   that
  \begin{equation}\label{key1}
    f(\dd)\geq \eps_0-\delta_0>0.
  \end{equation}
On the other hand, using \eqref{gradientflow1} and the constraint $|\dd|=1$, we deduce that
  \begin{equation*}
    \begin{split}
      &\quad(\p_t-\Delta)e(\dd)\\
      &=\nabla(\dd_t-\Delta\dd):\nabla\dd-|\nabla^2\dd|^2\\
      &=\nabla|\nabla\dd|^2\dd:\nabla\dd+
      |\nabla\dd|^2\nabla\dd:\nabla\dd+\nabla\p V(\dd):\nabla \dd-|\nabla^2\dd|^2\\
      &=|\nabla\dd|^4+ \nabla\p V(\dd):\nabla \dd-|\nabla^2\dd|^2,
    \end{split}
  \end{equation*}
    and
    \begin{equation*}
    (\p_t-\Delta)f(\dd)=\dd\cdot \qq |\nabla\dd|^2+\p V(\dd)\cdot \qq.
    \end{equation*}
 As a result,
    \begin{equation*}
    \begin{split}
      &(\p_t-\Delta)A(x,t)\\
      &=\frac{(\p_t-\Delta)e(\dd)}{f^2}-\frac{2e(\dd)(\p_t-\Delta)f}{f^3}
      +\frac{4\nabla e(\dd)\cdot\nabla f}{f^3}-\frac{6e(\dd)|\nabla f|^2}{f^4}\\
      &=\frac{|\nabla\dd|^4+\nabla\p V(\dd):\nabla \dd-|\nabla^2\dd|^2}{f^2}-\frac{|\nabla\dd|^4(\dd\cdot \qq)+|\nabla\dd|^2\p V(\dd)\cdot \qq}{f^3}\\
      &\qquad
      +\frac{4\nabla e(\dd)\cdot\nabla f}{f^3}-\frac{6e(\dd)|\nabla f|^2}{f^4}\\
      &=-\frac{\delta_0|\nabla\dd|^4}{f^3}+\frac{\nabla\p V(\dd):\nabla \dd}{f^2}-\frac{|\nabla^2\dd|^2}{f^2}-\frac{|\nabla\dd|^2\p V(\dd)\cdot \qq}{f^3}\\&\qquad
      +\frac{4\nabla e(\dd)\cdot\nabla f}{f^3}-\frac{6e(\dd)|\nabla f|^2}{f^4}\triangleq I_1+I_2,
    \end{split}
  \end{equation*}
  or simply
  \begin{equation}\label{eq:parabolic}
    (\p_t-\Delta)A(x,t)= I_1+I_2,
  \end{equation}
  where
  \begin{subequations}
  \begin{align}
  I_1=&-\frac{|\nabla^2\dd|^2}{f^2}
      +\frac{4\nabla e(\dd)\cdot\nabla f}{f^3}-\frac{6e(\dd)|\nabla f|^2}{f^4}, \\
      I_2=&-\frac{\delta_0|\nabla\dd|^4}{f^3}+\frac{\nabla\p V(\dd):\nabla \dd}{f^2}-\frac{|\nabla\dd|^2\p V(\dd)\cdot \qq}{f^3}.
  \end{align}
  \end{subequations}
  It remains to treat $I_1$ and $I_2$:
  \begin{equation}\label{estimatei1}
    \begin{split}
      I_1&=-\frac{|\nabla^2\dd|^2}{f^2}
      +\frac{2\nabla e(\dd)\cdot\nabla f}{f^3}-\frac{2e(\dd)|\nabla f|^2}{f^4}+\frac{2\nabla e(\dd)\cdot\nabla f}{f^3}-\frac{4e(\dd)|\nabla f|^2}{f^4}\\
      &\leq-\frac{|\nabla^2\dd|^2}{f^2}
      +\frac{2|\nabla \dd||\nabla^2 \dd||\nabla f|}{f^3}-\frac{|\nabla \dd|^2|\nabla f|^2}{f^4}+\frac{2\nabla A\cdot\nabla f}{f}\\
      &=-\(\frac{|\nabla^2\dd|}{f}
      -\frac{|\nabla \dd||\nabla f|}{f^2}\)^2+\frac{2\nabla A\cdot\nabla f}{f}\leq \frac{2\nabla A\cdot\nabla f}{f}.
    \end{split}
  \end{equation}
To treat  $I_2$, we employ \eqref{variationpotential} and deduce
   \begin{equation}\label{key3}
    \p V(\dd)\cdot \qq= \lambda^2(\mu-\mu^3)\geq 0,
  \end{equation}
  since $\mu=\dd\cdot \qq\in [0,1]$.
As   $\lambda=\lambda(t)$, we have
  \begin{equation}\label{key2}
  \nabla\p V(\dd):\nabla \dd= \lambda^2|\nabla \mu|^2-\lambda^2\mu^2 |\nabla \dd|^2.
  \end{equation}
  So  it follows from  \eqref{key2}, \eqref{key3} and $|\nabla\mu|\leq |\nabla\dd|$ that
  \begin{equation}\label{estimatei2}
 \begin{split}
 I_2&= \frac{(\mu-\delta_0)\lambda^2|\nabla\mu|^2+(\delta_0\mu-1)\mu\lambda^2|\nabla\dd|^2-\delta_0|\nabla\dd|^4}{f^3}\\
 &\leq \frac{\delta_0(\mu^2-1)\lambda^2|\nabla\dd|^2-\delta_0|\nabla\dd|^4}{f^3}\leq 0.
 \end{split}
 \end{equation}

 Now plugging \eqref{estimatei1} and \eqref{estimatei2} into \eqref{eq:parabolic} leads to
 \begin{equation}\label{eq:parabolic1}
    (\p_t-\Delta)A(x,t)\leq \frac{2\nabla A\cdot\nabla f}{f}.  \end{equation}
    If the maximum of $A(x,t)$ is achieved at $(x_1,t_1)\in \Sigma=\p\O\times (0,T)$, then by the  strong maximum principle,
    \begin{equation*}
    \p_\nu A(x_1,t_1)<0.
    \end{equation*}
    On the other hand, since $\p\O$ is convex, it follows from  \cite[pp. 162]{MR0482822} that
    \begin{equation*}
\p_\nu|\nabla\dd|^2\geq 0~\qquad \text{on}~\Sigma.
    \end{equation*}
    As a result
    \begin{equation*}
\begin{split}
 \p_\nu   A(x,t)=\frac 12 \p_\nu|\nabla\dd|^2 |\dd\cdot\qq-\delta_0|^{-2}
 +\frac{|\nabla\dd|^2}2\p_\nu|\dd\cdot\qq-\delta_0|^{-2}\geq 0 ~\text{on}~\Sigma.
\end{split}
    \end{equation*}
         So we obtain a contradiction, and thus the maximum must be achieved on $\O\times \{0\}$:
      \begin{equation*}
    A(x,t)\leq \sup_{x\in\O}\frac{|\nabla \dd_0|^2}{2 f^2(\dd_0)} \leq \frac{1}{2(\eps_0-\delta_0)^2}\sup_{x\in\O}|\nabla\dd_0|^2.\end{equation*}
    This implies  the desired result by choosing $\delta_0=\eps_0/2$.
   \end{proof}
With the aid of the above lemmas, we can  give the proof of Proposition \ref{mainwellpose}:
\begin{proof}[Proof of Proposition \ref{mainwellpose}]

  Since \eqref{gradientflow1} is a semi-linear parabolic system, the existence and uniqueness of the local in time solution follow from standard theory (see e.g.  \cite[Chapter 15]{MR2744149}): there exists $T_0>0$ such that
  \begin{equation*}
 \| \dd\|_{C^{2+\a,1+\frac{\a}{2}}(\overline{\O}\times (0,T_0),\S^2)}\leq C(T_0,\dd_0,\lambda).
  \end{equation*}
Applying Lemma \ref{notflee} gives
  \begin{equation}\label{lowerbound}
    \mu(x,t)=\dd(x,t)\cdot \qq\geq \eps_0> 0~\text{in}~  \O\times (0,T_0).
  \end{equation}
  In order to   extend the solution to every $T>0$, we need to bound $\|\dd\|_{C^{2+\a,1+\frac{\a}{2}}(\overline{Q})}$ in terms of   $\|\dd_0\|_{C^{2+\alpha}(\overline{\O})}$ up to a constant that is independent of  $T$, and this is a consequence of   Lemma \ref{gradientest}. More precisely,   the right hand side of \eqref{gradientflow1} is bounded  in $L^\infty(\overline{Q})$  by a constant depending on $\sup_{x\in\O}|\nabla\dd_0|, \epsilon_0,\lambda$ and $\Omega$ but not on  $T$. So parabolic regularity theory implies $\|\dd\|_{C^{1+\a,1/2+\alpha/2}(\overline{Q})}$ is bounded by a constant that is independent of $T$. Consequently the right hand side of  \eqref{gradientflow1} lies in $C^{\a,\a/2}(\overline{Q})$, and thus the Schauder's estimate  implies
  \begin{equation*}
    \|\dd\|_{C^{2+\a,1+\frac{\a}{2}}(\overline{Q})}\leq C,
  \end{equation*}
  where $C$ is independent of $T$. This completes the proof of existence of global in time classical solution to \eqref{gradientflow1}. The uniqueness of the solution follows from the standard energy method and \eqref{eq:2.4} follows from Lemma \ref{gradientest}.
\end{proof}

Now we turn to the estimate of the time derivative:
\begin{lemma}\label{jagermaxi}
Let $\dd_1,\dd_2:\O\times (0,T)\to \S^2$ be classical solutions of \eqref{gradientflow1} and
\begin{subequations}
\begin{align}
\psi &=1-\dd_1\cdot\dd_2,\quad\psi_i=1-\dd_i\cdot \qq,\quad w(t)=-\log(1-t),\label{maximum2}\\
\Phi &=\sum_{i=1}^2w\circ\psi_i=-\sum_{i=1}^2\log (1-\psi_i).\label{maximum3}
\end{align}
\end{subequations}
Then the operator
\begin{equation}
L(f)\triangleq\nabla\cdot(e^{-\Phi}\nabla f)-e^{-\Phi}\p_t f
\end{equation}
satisfies
  \begin{equation}\label{supersolu1}
    L(e^{\Phi}\psi)\geq 0.
  \end{equation}
\end{lemma}

\begin{proof}
Since $|\dd_i|=1$, the following formula will be frequently employed in the sequel:
\begin{equation}\label{simple1}
2\psi=2(1-\dd_1\cdot\dd_2)=|\dd_1-\dd_2|^2.
\end{equation}

We first prove the following inequality:
\begin{equation}\label{maximum1}
    L(e^{\Phi}\psi)\geq\left\{
    \begin{array}{ll}
        -\p_t \psi+\Delta \psi,~&x\in  \psi^{-1}(0),\\
        \displaystyle-\p_t \psi+\Delta \psi-\frac{|\nabla\psi|^2}{2\psi}+\psi\sum_{i=1}^2
        \frac{-\p_t\psi_i+\Delta\psi_i}{1-\psi_i},~& x\in \O\backslash\psi^{-1}(0).
    \end{array}
    \right.
  \end{equation}
  Direct computation shows
  \begin{equation}\label{maximum4}
    \begin{split}
    L(e^{\Phi}\psi)&=-\p_t\Phi\psi-\p_t \psi+\nabla\cdot(\nabla\Phi\psi+\nabla\psi)\\
      &=(-\p_t \Phi+\Delta\Phi)\psi+\nabla\Phi\cdot\nabla\psi-\p_t \psi+\Delta\psi.
    \end{split}
  \end{equation}
On $\psi^{-1}(0)$, it holds $\dd_1=\dd_2$ and thus
\begin{equation}
\nabla\psi=(\dd_1-\dd_2)\cdot\nabla (\dd_1-\dd_2)=0 \qquad\text{in}~\psi^{-1}(0).
\end{equation}
 These together with \eqref{psi est} below and \eqref{variationpotential} imply
 \begin{equation}
 L(e^{\Phi}\psi)=-\p_t \psi+\Delta \psi\geq 0 \qquad\text{in}~\psi^{-1}(0).
 \end{equation}
It remains  to consider the case when $x\in \Omega\backslash\psi^{-1}(0)$. It follows from the Cauchy-Schwartz inequality and the identity $w''=(w')^2$ that
  \begin{equation*}
    \begin{split}
      &(-\p_t \Phi+\Delta\Phi)\psi+\nabla\Phi\cdot\nabla\psi\\
    =&\sum_{i=1}^2\left[\psi(-\p_t\psi_i+\Delta\psi_i) w'\circ\psi_i +\psi|\nabla\psi_i|^2 w''\circ\psi_i+\nabla\psi_i\cdot\nabla\psi \,w'\circ\psi_i\right]\\
    \geq & \sum_{i=1}^2\left[\psi\(\frac{-\p_t\psi_i+\Delta\psi_i}{1-\psi_i}+|\nabla\psi_i|^2 w''\circ\psi_i\)-\psi(w'\circ\psi_i)^2|\nabla\psi_i|^2-\frac{|\nabla\psi|^2}{4\psi}\right]\\
    = &\psi \sum_{i=1}^2\frac{-\p_t\psi_i+\Delta\psi_i}{1-\psi_i}-\frac{|\nabla\psi|^2}{2\psi}
    \end{split}
  \end{equation*}
  This together with \eqref{maximum4} implies \eqref{maximum1}.

To proceed, we shall compute $-\p_t \psi+\Delta\psi$ using the first equation in \eqref{gradientflow1}
   \begin{equation}\label{psi est}
     \begin{split}
       &-\p_t \psi+\Delta\psi\\
       &=(\p_t\dd_1-\Delta\dd_1)\cdot\dd_2+\dd_1\cdot(\p_t\dd_2-\Delta\dd_2)-2\nabla\dd_1:\nabla\dd_2\\
       &=\(|\nabla \dd_1|^2+|\nabla \dd_2|^2\)\dd_1\cdot\dd_2+\p V(\dd_1)\cdot\dd_2+\p V(\dd_2)\cdot\dd_1-2\nabla\dd_1:\nabla\dd_2.
     \end{split}
   \end{equation}
   We also have
   \begin{equation}
       \psi\sum_{i=1}^2|\nabla\dd_i|^2=(1-\dd_1\cdot\dd_2)\(|\nabla \dd_1|^2+|\nabla \dd_2|^2\).
\end{equation}
Adding up the above two formulaes and using \eqref{simple1} yield:
   \begin{equation*}
     \begin{split}
      -\p_t\psi+ \Delta\psi+\psi\sum_{i=1}^2|\nabla\dd_i|^2 &=\p V(\dd_1)\cdot\dd_2+\p V(\dd_2)\cdot\dd_1+|\nabla(\dd_1-\dd_2)|^2\\
       &\geq\p V(\dd_1)\cdot\dd_2+\p V(\dd_2)\cdot\dd_1+\frac{\left|\nabla |\dd_1-\dd_2|^2\right|^2}{4|\dd_1-\dd_2|^2}\\
       &=\p V(\dd_1)\cdot\dd_2+\p V(\dd_2)\cdot\dd_1+\frac{\left|\nabla \psi\right|^2}{2\psi}.
     \end{split}
   \end{equation*}
By a similar calculation, we obtain
   \begin{equation*}
     \begin{split}
       \psi \sum_{i=1}^2\frac{-\p_t\psi_i+\Delta\psi_i}{1-\psi_i}=\psi\sum_{i=1}^2|\nabla\dd_i|^2+\psi\sum_{i=1}^2\frac{\p V(\dd_i)\cdot \qq}{\dd_i\cdot \qq}.
     \end{split}
   \end{equation*}
   Adding up the above two inequalities and then using \eqref{variationpotential} lead to
        \begin{equation*}
     \begin{split}
       -\p_t\psi&+\Delta\psi-\frac{|\nabla\psi|^2}{2\psi}+\psi \sum_{i=1}^2\frac{-\p_t\psi_i+\Delta\psi_i}{1-\psi_i} \\
       \geq&\p V(\dd_1)\cdot\dd_2+\p V(\dd_2)\cdot\dd_1+\psi\sum_{i=1}^2\frac{\p V(\dd_i)\cdot \qq}{\dd_i\cdot \qq}\\
       =&\l^2(\dd_1-\dd_2)^2-\l^2(\qq\cdot(\dd_1-\dd_2))^2\geq 0.
     \end{split}
   \end{equation*}
   This together with \eqref{maximum1} leads to the desired result.
    \end{proof}
    The above lemma   implies  the estimate of  the time  derivative of \eqref{gradientflow1}.
    \begin{prop}\label{eq:timeest}
      Let $\dd\in C^{2+\a,1+\frac\a 2}(\overline{Q})$ be a solution of \eqref{gradientflow1} with
      \begin{equation}
      \eps_1\triangleq \inf_{x\in\O}\dd(x,T_1)\cdot \qq> 0,
      \end{equation}
      for some $T_1\in (0,T)$.
Then for any $T_2\in (T_1,T)$, the following inequality holds:      \begin{equation}
        \sup_{\O\times [T_1,T_2]}|\p_t \dd(x,t)|\leq \eps_1^{-1}\sup_{\O}|\p_t \dd(x,T_1)|.
      \end{equation}
\end{prop}
\begin{proof}
For any $t_0\in (T_1,T_2)$, there exists $h_0>0$ such that $t_0+h_0<T_2$. For any $h\in (0,h_0)$, the functions $\dd_1(x,t)\triangleq\dd(x,t)$ and $\dd_2(x,t)\triangleq\dd(x,t+h)$ are  well defined on $\O\times [T_1,t_0]$ and are both  solutions to \eqref{gradientflow1}.
The boundary condition of \eqref{gradientflow1} implies
  \begin{equation*}
    \sup_{\p\O\times[T_1,t_0]} \p_{\nu}\(e^{\Phi}\psi(x,t)\)=0.
  \end{equation*}
  So   \eqref{supersolu1} in Lemma \ref{jagermaxi} together with the  maximum principle imply
  \begin{equation}\label{MP1}
    \sup_{\O\times (T_1,t_0)}e^{\Phi}\psi(x,t)\leq \sup_{\O\times{\{T_1\}}} e^{\Phi}\psi(x,t).
  \end{equation}
    Moreover Lemma \ref{notflee} implies
  \begin{equation*}
    0<\eps_1^2\leq e^{-\Phi}=(\dd_1\cdot \qq)(\dd_2\cdot \qq)\leq 1.
  \end{equation*}
  Consequently,
  \begin{equation*}
    \sup_{\O\times(T_1,t_0)} |\dd(x,t+h)-\dd(x,t)|^2\leq \eps_1^{-2} \sup_{\O}  |\dd(x,T_1+h)-\dd(x,T_1)|^2.
  \end{equation*}
  Dividing the above estimate by $h^2$ and taking  $h\to 0$ lead to the desired estimate since $t_0$ is arbitrary.
\end{proof}
We end this   section by studying the regularizing effect of  \eqref{gradientflow1} when $\lambda(t)\equiv 0$:
  \begin{equation}\label{gradientflow2}
\left\{
\begin{array}{rll}
 \p_t\dd-\Delta \dd &=|\nabla \dd |^2\dd ,&~\text{in}~
 \Omega\times [0,T),\\
\p_\nu \dd&=0,&~\text{on}~\p\Omega\times [0,T).
\end{array}
\right.
  \end{equation}

\begin{prop}\label{smoothing effect}
Let $\dd$ be a classical solution of   \eqref{gradientflow2}   with
\begin{equation}\label{lipschitz bound}
\|(\p_t\dd,\nabla\dd)\|_{L^\infty([0,T]\times\Omega)}\leq M,
\end{equation}
 then  there exists a constant $C=C(M,\Omega)>0$   such that
\begin{equation}\label{improved regularity}
\begin{split}
&\|\dd(\cdot,T_2)-\qq\|_{W^{1,\infty}(\Omega)}\\ &\leq C \(\tfrac 1{T_2-T_1}+1\)\|\dd-\qq\|_{L^\infty([T_1,T_2]\times \Omega)},~\forall~  0\leq T_1< T_2\leq T.
\end{split}
\end{equation}
\end{prop}
\begin{proof}
Without loss of generality, we can just work with the case when   $T_1=0,T_2=T$.
We shall estimate the difference  $\overline{\dd}\triangleq \dd-\qq$, which fulfills
\begin{equation}\label{differenceequ}
\p_t \overline{\dd}-\Delta \overline{\dd}=|\nabla\overline{\dd}|^2 (\overline{\dd}+\qq).
\end{equation}
 
We first deduce from \eqref{lipschitz bound} that
\begin{equation}
\|\Delta \dd\|_{L^\infty(\Omega\times (0,T))}\leq \|\p_t \dd\|_{L^\infty(\Omega\times (0,T))}+\||\nabla\dd|^2\dd\|\|_{L^\infty(\Omega\times (0,T))} \leq M+M^2.
\end{equation}
 This together with elliptic regularity leads to
\begin{equation}\label{lipschitz bound1}
\|\dd\|_{L^\infty(0,T;W^{2,p}(\Omega))}\leq C(M,p),~\forall p>3.
\end{equation}
As a result, we obtain the following estimate of the nonlinear terms:
\begin{equation}
\left\||\nabla \dd|^2\dd\right\|_{L^\infty(0,T;W^{1,p}(\Omega))}\leq C(M,p),~\forall p>3.
\end{equation}
Moreover, we  infer from \eqref{lipschitz bound1} and the  Gagliardo-Nirenberg interpolation inequality that
\begin{equation}\label{improve1}
\|\overline{\dd}\|_{L^\infty(0,T;W^{1,p}(\Omega))}\leq C(M,p)\|\overline{\dd}\|^{1/2}_{L^\infty(\Omega_T)},~\forall p>3.
\end{equation}
To proceed,  we   write \eqref{differenceequ} in terms of the heat semigroup
\begin{equation}\label{semi expansion}
\begin{split}
\overline{\dd}(\cdot,t)=e^{t\Delta} \overline{\dd}(\cdot,0)+ \int_0^t e^{(t-\tau)\Delta}\(|\nabla \overline{\dd}|^2(\overline{\dd}+\qq)\)(\cdot,\tau)\, d\tau.
\end{split}
\end{equation}
The first term on the right hand side above   can be estimated by Sobolev embedding and  semi-group property
\begin{equation}
\|e^{t\Delta} \overline{\dd}(\cdot,0)\|_{W^{1,\infty}(\Omega)}\leq C\|e^{t\Delta} \overline{\dd}(\cdot,0)\|_{W^{2,4}(\Omega)}\leq C t^{-1} \|\overline{\dd}|_{t=0}\|_{L^{4}(\Omega)}.
\end{equation}
Regarding the second term on the right hand side of \eqref{semi expansion}, we employ   \eqref{heat est1} and \eqref{improve1}  to  yield
\begin{equation}
\begin{split}
&\left\|\int_0^t  e^{(t-\tau)\Delta}\(|\nabla \overline{\dd}|^2(\overline{\dd}+\qq)\)(\cdot,\tau)\, d\tau\right\|_{C([0,T];W^{1,\infty}(\Omega))}\\
&\leq  \left\||\nabla \overline{\dd}|^2(\overline{\dd}+\qq)\right\|_{L^p(\Omega\times (0,T))} \leq C(M,p)\|\overline{\dd}\|_{L^\infty(\Omega\times (0,T))}, \end{split}
\end{equation}
with  $p>5$. So we prove \eqref{improved regularity}.
\end{proof}


\section{Proof  of Theorem \ref{mainthm}}
\setcounter{equation}{0}
We first observe that for any $\qq,\mathbf{p}\in\S^2$, there exists $\mathbf{p}_1,\mathbf{p}_2\in\S^2$ such that
\begin{equation}\label{intermedia}
\qq\cdot\mathbf{p}_1>0,\, \mathbf{p}_1\cdot\mathbf{p}_2>0,\, \mathbf{p}_2\cdot \mathbf{p}>0.
\end{equation}
Actually, we can simply choose $\mathbf{p}_1,\mathbf{p}_2$ which trisect the angle $\theta\in [0,\pi]$ expanded by $\qq$ and $\mathbf{p}$.
We shall show that for any initial   state  $\dd_0\in C^{2+\a}(\overline{\O},\S^2)$ satisfying  \eqref{hemi}, there is a control of form \eqref{eq:1.6} such that $\dd(\cdot,\tfrac T4)=\qq$. Once this special case  is done, we can apply it on the interval $[\tfrac T4,\tfrac{T}2]$ to have $\dd(\cdot,\tfrac{T}2)=\mathbf{p}_1$. Then again on $[\tfrac T2,\tfrac{3T}4]$, we can achieve  $\dd(\cdot,\tfrac{3T}4)=\mathbf{p}_2$ and finally  $\dd(\cdot,T)=\mathbf{p}$.   This process is feasible due to   the rotational invariance  of \eqref{gradientflow}.  More precisely, if $(\dd,\HH)$ satisfies \eqref{gradientflow}, so does $(\mathcal{R}\dd,\mathcal{R}\HH)$ for every orthogonal matrix $\mathcal{R}$.

To show the controllability to $\qq$ in $[0,\tfrac T4]$, we denote
$T_0=\frac T{24}$ and
choose the control
\begin{equation}\label{control 1}
\HH(x,t)=\gg(t)\triangleq \lambda(t) \qq,~\forall t\in [0,5T_0],
\end{equation}
where  $\lambda(t)\in C^1([0,  5T_0])$ is non-negative  so  that
\begin{equation}\label{large control}
\lambda(t)=\left\{
\begin{array}{ll}
\Lambda\in\R^+ &~\text{when}~t\in [2T_0, 3T_0],\\
0&~\text{when}~t\in [0,T_0]\cap [4T_0,5T_0].
\end{array}
\right.
\end{equation}
 By choosing the constant $\Lambda$  sufficiently large,   the initial data can  be driven  to a neighborhood of the ground state $\qq$ within   $[0,5T_0]$ such that Proposition \ref{conlinear} can be applied for $t\in [5T_0,6T_0]$.
More precisely, by the assumption \eqref{hemi} and   Proposition \ref{mainwellpose}, there exists  a unique solution $\dd\in C^{2+\a,1+\frac{\a}{2}}(\overline{\O}\times[0, 5T_0])$ to \eqref{gradientflow1} with initial data $\dd_0$ such that
\begin{subequations}
        \begin{align}
        &\sup_{\O\times [0,5T_0]}|\nabla\dd(x,t)|\leq \frac{2}{\eps_0}\sup_{x\in\O}|\nabla\dd_0|,\label{initialdata1}\\
&\mu(x,t)=      \dd(x,t)\cdot \qq\geq  \eps_0> 0,\quad \forall(x,t)\in\O\times [0,5 T_0].\label{equ of mu}
        \end{align}
    \end{subequations}
By our choice of $\lambda(t)$, no control is applied in $[0,T_0]$. So Proposition \ref{mainwellpose} implies
\begin{equation}\label{large control1}
\sup_{\O\times \{T_0\}}|\p_t \dd(x,t)|\leq C(\dd_0).
\end{equation}
Thanks to Proposition \ref{eq:timeest}, $\|\p_t \dd(\cdot, t)\|_{L^\infty(\O)}$ will not increase in $t$ due to the presence of the control \eqref{large control}  with a large $\Lambda$. So the combination of  \eqref{large control1} and \eqref{initialdata1} leads to
\begin{equation}
\sup_{\O\times [T_0, 5T_0]}\(|\nabla \dd(x,t)|+|\p_t \dd(x,t)|\)\leq C(\dd_0,\epsilon_0,\O),
\end{equation}
where  $C(\cdot )$  is independent of $\Lambda$. Since no control is applied in $[4T_0,5T_0]$, we employ Proposition \ref{smoothing effect} and deduce
\begin{equation}\label{improved regularity1}
\|\dd(\cdot,5T_0)-\qq\|_{W^{1,\infty}(\Omega)}\leq C(\dd_0,\epsilon_0,\O) T_0^{-1}\|\dd-\qq\|_{L^\infty([4T_0,5T_0]\times \Omega)}.
\end{equation}
On the other hand,  it follows from  \eqref{equ of mu} and \eqref{large control} that $\mu$ satisfies
  \begin{equation}
    \p_t \mu-\Delta \mu=|\nabla \dd|^2 \mu+ \Lambda^2(\mu-\mu^3),~\text{for}~t\in [2T_0,3T_0].
  \end{equation}
If we denote $\phi=1-\mu$, since $\mu\in [\epsilon_0,1]$ with  $\epsilon_0\in (0,1)$, we have
    \begin{equation}\label{compare1}
    \p_t \phi-\Delta \phi\leq   -\Lambda^2(1-\phi)\phi(2-\phi)\leq -\Lambda^2\epsilon_0\phi,~\text{for}~t\in [2T_0,3T_0].
  \end{equation}
Applying the  comparison  principle yields  the decay
  \begin{equation*}
 1-\dd(x,t)\cdot\qq= \phi(x,t)\leq e^{-\Lambda^2 \epsilon_0 (t-2T_0)},~\forall t\in [2T_0,3T_0].
  \end{equation*}
For any  $T_0>0$ and   $\Lambda>0$, we set
\begin{equation}\label{lambda choice}
\epsilon_4=e^{-\Lambda^2\epsilon_0 T_0}.
\end{equation}
Then there holds
  \begin{equation}\label{push1}
    0\leq 1-\dd(x,3T_0)\cdot\qq \leq \epsilon_4,~\forall x\in\O.
  \end{equation}
For  $t\in [3T_0,5T_0]$, instead of having \eqref{compare1}, we have $\p_t\phi-\Delta \phi\leq 0$, and thus the maximum principle implies
  \begin{equation}
 0\leq  1-\dd(x,t)\cdot\qq \leq \sup_{\O}\(1-\dd(x,3T_0)\cdot\qq\) \leq \epsilon_4,~\forall t\in [3T_0,5T_0].
  \end{equation}
  This combined with \eqref{improved regularity1} leads to
  \begin{equation}\label{improved regularity2}
\|\dd(\cdot,5T_0)-\qq\|_{W^{1,\infty}(\Omega)}\leq C(\dd_0,\epsilon_0,\O) T_0^{-1}\epsilon_4.
\end{equation}
So for any  $T_0\in (0,1)$,  by choosing a  sufficiently large  $
\Lambda$ in \eqref{lambda choice}, we shall have $\epsilon_4$ being sufficiently small so that    $\vv_0(x)\triangleq \mathbf{\Psi}^{-1}(\dd(x,5T_0))$ will satisfy   \eqref{initialdata} where $\mathbf{\Psi}$ is  the stereographic projection  defined by \eqref{stereograph}. Then we consider the control system \eqref{noncontrolstere} on $[5T_0,6T_0]$.     This is the second stage of the control process, where  we can apply Proposition \ref{conlinear}
 to obtain a control $\ff\in L^\infty(\O\times (5T_0,6T_0)$ such that $\vv(\cdot,6T_0)\equiv 0$.
 According to Proposition \ref{equiv3} and Lemma \ref{implicitfunc}, we have $\dd=\mathbf{\Psi}(\vv), \HH=\HH(\ff,\vv)$ satisfying \eqref{gradientflow} on $\O\times (5T_0,6T_0)$, and $\dd(\cdot,6T_0)=\qq$. This completes the proof    of Theorem \ref{mainthm}.

\section*{Acknowledgments}
The author would like to thank Professor Xu Zhang for helpful discussions.  

\end{document}